\documentclass[12pt,oneside,british,a4paper]{amsart}

\usepackage[T1]{fontenc}
\usepackage[utf8]{inputenc}
\usepackage{latexsym,cite,euscript,amscd,bm,mathabx}
\usepackage{hyperref}
\usepackage[dvipsnames]{xcolor}
\usepackage{amstext,amsthm,amssymb}
\usepackage{xparse,mathrsfs,mathtools}
\usepackage{csquotes}
\usepackage{bookmark}
\usepackage[capitalise,nameinlink]{cleveref}
\usepackage[british]{babel}
\usepackage{lmodern,microtype}
\usepackage{enumitem}
\usepackage[a4paper]{geometry}

\usepackage[norefs,nocites]{refcheck}
\makeatletter
\newcommand{\refcheckize}[1]{%
	\expandafter\let\csname @@\string#1\endcsname#1%
	\expandafter\DeclareRobustCommand\csname relax\string#1\endcsname[1]{%
		\csname @@\string#1\endcsname{##1}\wrtusdrf{##1}}%
	\expandafter\let\expandafter#1\csname relax\string#1\endcsname
}
\makeatother
\refcheckize{\cref}
\refcheckize{\Cref}

\usepackage{xifthen,ifthen}
\makeatletter
\newcounter{@ToDo}
\newcommand{\todo@helper}[1]{%
	({\color{blue}TODO~\arabic{@ToDo}: {#1\@addpunct{.}}})%
}
\newcommand{\todo}[1]{\stepcounter{@ToDo}%
	\relax\ifmmode\text{\todo@helper{#1}}%
	\else\todo@helper{#1}\fi%
}

\hypersetup{
	unicode=true,%
	colorlinks,linkcolor=blue,anchorcolor=blue,citecolor=blue,%
	pdfdisplaydoctitle=true,%
	pdfpagemode=UseOutlines,%
	breaklinks=true,%
	hidelinks,
	pdfencoding=unicode,psdextra,
	pdfcreator={},
	pdfborder={0 0 0}
}

\newtheorem{thm}{Theorem}
\newtheorem{lem}[thm]{Lemma}
\newtheorem{cor}[thm]{Corollary}
\newtheorem{prop}[thm]{Proposition}
\theoremstyle{remark}

\DeclarePairedDelimiter\parentheses{\lparen}{\rparen}
\DeclarePairedDelimiter\floor{\lfloor}{\rfloor}
\DeclarePairedDelimiter\ceil{\lceil}{\rceil}
\DeclarePairedDelimiter\braces{\lbrace}{\rbrace}
\DeclarePairedDelimiter\abs{\lvert}{\rvert}

\newcommand{\NN}{\mathbb{N}}
\newcommand{\ZZ}{\mathbb{Z}}

\newcommand{\FF}{\mathbb{F}}

\newcommand{\smallDelta}{\delta_0}
\newcommand{\scrI}{\mathscr{I}}
\newcommand{\fS}{\mathfrak{S}}
\newcommand{\mand}{\qquad\mbox{and}\qquad}

\newcommand{\dd}[1]{\mathop{\mathrm{d}#1}}
\DeclareMathOperator{\eOpname}{e}
\NewDocumentCommand\e{ s O{} m o }{
	\IfBooleanTF{#1}{%
		\eOpname_{#2}\parentheses*{#3}%
	}{%
	\IfNoValueTF{#4}{%
		\eOpname_{#2}\parentheses{#3}
	}{%
		\eOpname_{#2}\parentheses[#4]{#3}
	}}%
}

\numberwithin{thm}{section}
\numberwithin{equation}{section}
\crefformat{equation}{(#2#1#3)}
\crefrangeformat{equation}{(#3#1#4) to~(#5#2#6)}

\title{Kloosterman sums with twice-differentiable functions}
\date{\today}


\begin{document}
\author{Igor~E.~Shparlinski}
\address{Igor~E.~Shparlinski\\%
	School of Mathematics and Statistics\\%
	The University of New South Wales\\%
	Sydney NSW~2052\\%
	Australia}
\email{igor.shparlinski@unsw.edu.au}

\author{Marc~Technau}
\address{Marc~Technau\\%
	Institute of Analysis and Number Theory\\%
	Graz University of Technology\\%
	Kopernikusgasse~24\\%
	8010~Graz\\%
	Austria}
\email{mtechnau@math.tugraz.at}


\begin{abstract}
	We bound Kloosterman-like sums of the shape
	\[
		\sum_{n=1}^N \exp(2\pi i (x \floor{f(n)}+ y \floor{f(n)}^{-1})/p),
	\]
	with integers parts of a real-valued, twice-differentiable function $f$ is  satisfying a certain limit condition on $f''$, and $\floor{f(n)}^{-1}$ is meaning inversion modulo~$p$.
	As an immediate application, we obtain results concerning the distribution of modular inverses inverses $\floor{f(n)}^{-1} \pmod{p}$. 	The results apply, in particular, to  Piatetski-Shapiro sequences $ \floor{t^c}$ with $c\in(1,\frac{4}{3})$.  
	The proof is an adaptation of an argument used by Banks and the first named author in a series of papers from~2006 to~2009.
\end{abstract}
\maketitle

\section{Introduction and main result}

\subsection{Background and motivation}

The \emph{Piatetski-Shapiro sequence} associated with $c\in(1,2)$ is defined by $(\floor{n^c})_{n\in\NN}$, where $\floor{x} = \min\braces{ m\in\ZZ : m\leq x }$ denotes the floor function. Such sequences are named in honour of Pyatecki\u{\i}-\v{S}apiro~\cite{pyateckii-sapiro1953on-the-distribution}\footnote{Nowadays his name is usually spelled as Piatetski-Shapiro.} who, at the suggestion of A.~O.~Gelfond, has proved the following prime number theorem:
\begin{equation}\label{eq:Piatetski-Shapiro:PNT}
	\#\braces{ n \leq N:~\floor{n^c} \text{ prime} }
	= (1 + o(1)) \frac{N}{c \log N}
	\qquad \text{as } N\to\infty
\end{equation}
for $c$ in the range $1<c<\frac{12}{11}$. Such a result may be viewed as an intermediate step to tackling the problem of investigating the number of primes represented by a fixed quadratic polynomial; consider for instance Landau's famous problem of ascertaining whether $n^2+1$ is prime for infinitely many $n\in\NN$. Informally speaking, the upper bound for the exponents $c$ for which one is able to establish~\cref{eq:Piatetski-Shapiro:PNT} measures the progress towards quadratic polynomials. To date, the largest admissible $c$-range seems to be $1<c<\frac{2817}{2426}$ due to Rivat and Sargos~\cite{rivat2001nombres-premiers} (see also the references to the previous record holders they give in their paper).
Naturally, also lower bound sieves have been employed, and the corresponding current record is a version 
of~\cref{eq:Piatetski-Shapiro:PNT} with a lower bound of the right order of magnitude instead of an asymptotic formula and $1<c<\frac{243}{205}$ due to Rivat and Jie~\cite{rivat2001prime-numbers}.

Investigations into arithmetic properties of Piatetski-Shapiro sequences are not confined to studying prime values; they have been studied with respect to various other topics, including, but not limited to, the following:
\begin{itemize}
	\item smooth, rough, and square-free numbers~\cite{baker2013,akbal2018,akbal2017},
	\item almost-primes~\cite{baker2014,banks2016},
	\item additive problems \cite{mirek2015,akbal2016,lu2018},
	\item intersection with special sequences~\cite{banks2009multiplicative,banks2008,baker2015,akbal2015,liu2017},
	and
	\item digital expansions~\cite{mauduit1995repartition,spiegelhofer2014piatetski-shapiro,muellner2017normality}.
\end{itemize}
For a broader picture of the scope of each topic, we refer the interested reader to the references within the cited items.

Here we study a question of distribution of modular 
inverses modulo a prime $p$ of Piatetski-Shapiro sequences and, in fact, 
more general sequences. Our motivation comes from~\cite{technau2018modular-hyperbolas}, where the distribution of 
inverses modulo a prime $p$ of {\it Beatty sequences\/} is considered.
Furthermore, since this question immediately leads 
us to a problem of estimating \emph{Kloosterman-like} sums with Piatetski-Shapiro sequences, this serves as an
additional motivation. Indeed,  this is an additive analogue of the results from~\cite{banks2009multiplicative}, which concern
bounds for character sums of the form  
\begin{equation}\label{eq:CharacterSum}
	\sum_{n=1}^N \chi(\floor{f(n)})
	\qquad (N\geq 1),
\end{equation}
where $f$ is a real-valued, twice-differentiable function satisfying a certain limit condition on $f''$ (see~\cref{eq:2ndDerivativeGrowth} below), $\chi$ is a non-trivial multiplicative character modulo a prime $p$, which is assumed to be in a suitable range with respect to $N$.
The results from~\cite{banks2009multiplicative} apply, in particular, to power functions $f(t) = t^c$ with $c\in(1,\frac{4}{3})$ and apply, amongst other things, to bounding the least quadratic non-residue in Piatetski-Shapiro sequences (see also~\cite{banks2008,baker2015}).

\subsection{Notation}

Throughout the paper, the notation $U = O(V)$, 
$U \ll V$ and $ V\gg U$  are equivalent to $|U|\leq c|U| $ for some positive constant $c$, 
which, may occasionally depend on the function $f$ and on the small positive real 
parameter $\varepsilon$ and on the positive integer parameter $k$.  We use subscripts to indicate such 
dependencies.

We always use $p$ to denote a  prime number and then for  $x,y,u\in\FF_p = \ZZ/p\ZZ$, we define
\begin{equation}\label{eq:psi:definition}
	\psi_{x,y;\,p}(u) = \begin{cases}
		\e[p]{xu+yu^{-1}} & \text{if } u\neq 0, \\
		0 & \text{otherwise},
	\end{cases}
\end{equation}
where  $\e[p]{a} = \exp(2\pi i a / p)$ and the calculation inside the argument of $\e[p]{\,\cdot\,}$ is to be performed in $\FF_p$.

\subsection{Main results} 

In this paper, we first outline how to adapt the arguments from~\cite{banks2009multiplicative} to bound exponential sums with $\psi_{x,y;\,p}(\floor{f(n)})$ which immediately provide some non-trivial information about the distribution 
of modular inverses $\floor{f(n)}^{-1} \pmod p$.  
We remark that   for the scenario 
of~\cite{banks2006nonresidues,banks2006short-character-sums,banks2007prime-divisors},  
any non-trivial bound on the character sum implies the desired result about the distribution of quadratic 
non-residues.
We also use the opportunity to provide 
slightly more precise information about the 
dependence of our saving $\delta$ on the parameter $\kappa$, characterising the growth of $f$. In particular, 
the explicit formula~\cref{eq:delta} below shows that $\delta$ is a monotonic function  of $\varepsilon$ and $\kappa$.

Our main result may be stated as follows:
\begin{thm}
	\label{thm:main:result}
	Let $\frac{2}{3} < \kappa < 1$, and let $f$ be a real-valued, twice-differentiable function such that
	\begin{equation}\label{eq:2ndDerivativeGrowth}
		\lim_{t\to\infty} \frac{\log f''(t)}{\log t} = -\kappa.
	\end{equation}
	Then, for a sufficiently small $\varepsilon>0$, for all $\psi_{x,y;\,p}$, with $p$ prime and $x,y,u\in\FF_p$, as in~\cref{eq:psi:definition} and all integers $N$ in the range
       \begin{equation}\label{eq:ineq N}
		p^{1/(2\kappa)+\varepsilon} \leq N < p^{1/(2-\kappa)},
	\end{equation}
	the uniform bound
	\begin{equation}\label{eq:main:result}
		\sum_{n=1}^N \psi_{x,y;\,p}(\floor{f(n)})
		\ll_{\varepsilon, f} Np^{-\delta}
	\end{equation}
	holds with
	\begin{equation}\label{eq:delta}
		\delta = 2^{-11} \varepsilon^2 \kappa^4.
	\end{equation}
\end{thm}

For an interval $\scrI = [K+1, K+H-1]\subseteq [0,p-1] $ 
with integers $K$ and $H$, we denote by $T_f(N,\scrI) $ the number 
of positive integers $n \le N$ for which the smallest positive residue 
$\floor{f(n)}^{-1} \pmod p$ falls in $\scrI$, that is
\[
	T_f(N,\scrI) = \braces*{ n:~ \floor{f(n)}^{-1} \equiv h \pmod p,\, 1\le n \le N,\, h \in \scrI }.
\]

\begin{cor}\label{cor:DistrInt}
	On the hypotheses of \cref{thm:main:result}, we have
		\[
		T_f(N,\scrI)  = \frac{HN}{p} + O_{\varepsilon, f}\parentheses{Np^{-\delta}\log p}.
	\]
where $\delta$ is given by~\cref{eq:delta}. 
\end{cor}

\begin{cor}\label{cor:Exist}
	Let $f$ be a real-valued, twice-differentiable function such that~\cref{eq:2ndDerivativeGrowth} holds.
	There exists a constant $\xi>0$ which depends only on $\kappa$,  such that for 
	\[
		p\ge H,N \mand HN \ge p^{2-\xi}
	\]
	we have $T_f(N,\scrI) > 0$.   
\end{cor}
Note that \cref{cor:Exist} is an analogue of \cite[Theorem~5.1]{technau2018modular-hyperbolas}, where a  result of this type is given for Beatty sequences.

\section{Outline of the argument}

\subsection{Preliminaries}
As large parts of the arguments in~\cite{banks2009multiplicative} essentially carry over verbatim to the proof of \cref{thm:main:result}, we choose not to repeat them here in full detail.
Instead, we give an informal description of the underlying argument.
The argument ultimately relies on a bound for certain double sums, \cite[Lemma~4.1]{banks2009multiplicative}, and it is this bound which we need to adapt to our setting,
A proof of this adapted bound, which is contained in our \cref{cor:DoubleSumBound-eps} below, is then carried out in full detail in \cref{sec:DoubleSumBound}.
Some further details pertaining to the explicit value of $\delta$ given in~\cref{thm:main:result} are contained in \cref{sec:ExplicitLemma,sec:main proof}.

\subsection{Reduction from long sums to sums over short intervals}
To get started, consider the sum
\begin{equation}\label{eq:the:sum}
	\sum_{n=1}^N \psi_{x,y;\,p}(\floor{f(n)})
\end{equation}
from \cref{thm:main:result}. 
We now fix some constant $c>0$ (whose final choice depends on $\kappa$ and $\varepsilon$), to satisfy the 
inequalities
\begin{equation}\label{eq:N:and:p:relations:i_to_iii}
	N^{1-\kappa/2} \ge p^{3c}, \qquad
	N^{\kappa/2-c} \ge p^{3c}, \qquad
	N\ge p^{1/2+3c},
\end{equation}
as well as
\begin{equation} \label{eq:N:and:p:relations:iv}
	N^{\kappa-c} \ge p^{1/2+3c}.
\end{equation}
Next, for some parameter $R$, the sum~\eqref{eq:the:sum} can be decomposed into $R$ sum over a small initial segment, $n\leq Np^{-c}$ and  $R$ \enquote{short} sums over $K_j<n\leq K_{j-1}$ with  
$j=1,\ldots,R$ and numbers $N = K_0 > K_1 > \ldots > K_R = Np^{-c}$ satisfying
\begin{equation}\label{eq:short int}
	K_{j-1}-K_j = \Delta K_j, \qquad  j=1,\ldots,R. 
\end{equation}
for some parameter $\Delta$.

Concerning the sum over the initial segment, already the trivial estimate is satisfactory, as it is within the bound~\cref{eq:main:result} of \cref{thm:main:result}.
The remaining short sums are then treated with~\cref{lem:ShortSumBound} below. Of course, to do this in the first place, the parameters $R$ and $\Delta$ have to be chosen appropriately. The specific choice used in~\cite{banks2009multiplicative}, which also works in the setting of this paper, is (with the above choice of $c$)
\[
	R = \floor*{ N^{1+c-\kappa}\log^2 p }\mand \Delta = p^{c/R}-1, 
\]
and then
\[
	K_j=p^{-jc/R}N, \qquad  j=0,\ldots,R. 
\]
One now verifies that~\cref{eq:short int} holds. 
We plainly refer to~\cite{banks2009multiplicative} for the technical details. 

\subsection{Reduction from sums over short intervals to double sums} 
For the short sums, one can use the following result: (This is already adapted to our setting; see \cite[Lemma~5.1]{banks2009multiplicative} for the corresponding character sum variant.)

\begin{lem}
	\label{lem:ShortSumBound}
	Fix $\varepsilon >0$ and $\frac{2}{3} < \kappa < 1$. Let $f$ be a real-valued, twice-differentiable function satisfying~\cref{eq:2ndDerivativeGrowth}.
	Then, for all $\psi_{x,y;\,p}$, with $p$ prime and $x,y,u\in\FF_p$, as in~\cref{eq:psi:definition}, and all real numbers $K, L$ that satisfy the inequalities
	\begin{equation}\label{eq:ShortSumBound:Ranges}
		K^{\kappa-\varepsilon} \geq  L \geq  K^{\kappa/2} p^{\varepsilon }, \qquad
		K \leq p^{1/(2-\kappa)}, \qquad
		L \geq p^{1/2+\varepsilon },
	\end{equation}
	the uniform bound
	\[
		\sum_{K < n\leq K+L} \psi_{x,y;\,p}(\floor{f(n)})
		\ll_{\varepsilon ,f} L p^{-\smallDelta}
	\]
	holds with some constant $\smallDelta>0$ that may only depend on $\varepsilon$ and $\kappa$.
\end{lem}

Next, we sketch the idea of the proof of \cref{lem:ShortSumBound}. Trivially, for every integer $h\geq 0$, one has
\[
	\sum_{K < n\leq K+L} \psi_{x,y;\,p}(\floor{f(n)})
	= \sum_{K < n\leq K+L} \psi_{x,y;\,p}(\floor{f(n+h)}) + O(h).
\]
Therefore, upon averaging over all $h\leq H$ for some parameter $H$,
\begin{align}\label{eq:AveragedSum}
	\sum_{K < n\leq K+L} \psi_{x,y;\,p}(\floor{f(n)})
	= \frac{1}{H} \sum_{h=1}^{H} \sum_{K < n\leq K+L} \psi_{x,y;\,p}(\floor{f(n+h)}) + O(H).
\end{align}
To progress further, one would like to separate the variables $n$ and $h$ in the argument of $\psi_{x,y;\,p}$. This is achieved using the following formula:
\[
	f(n+h)  
	= f(n) + h f'(K) + I_{n,h} + J_{n,h}, 
\]
where  
\[
	{I_{n,h} = h \int_K^n f''(u) \dd{u}} \mand {J_{n,h} = \int_n^{h+n} f''(u) (h+n-u) \dd{u}}.
\]

In view of~\cref{eq:2ndDerivativeGrowth}, the relevant integrals can be seen to be acceptably small provided that $p$ is large enough. In particular, the last observation also crucially relies upon $n$ being not too large with respect to $K$,  which is the case, because we are dealing with short sums. We now suppose that 
\[0\leq I_{n,h} + J_{n,h} < 1.
\] 
However, there still is a complication with this approach arising through the floor function. Indeed, for any $\xi_0\in\lbrack 0,1\rparen$,
\begin{equation}\label{eq:carry}
	\floor{ f(n+h) }
	= \floor{f(n)
	+ \xi_0}
	+ \floor{hf'(K) - \xi_0}
	+ \eta_{n,h,\xi_0},
\end{equation}
with some undesired correction term $\eta_{n,h} \in \braces{0,1,2}$ depending on \emph{both} $n$ and $h$ (and $\xi_0$). To get around this, one restricts the averaging to only those $h\leq H$ for which the fractional part of $hf'(K)-\xi_0$ is small. By a clever choice of $\xi_0$, using the pigeonhole principle, one can ensure that the set of $h\leq H$ having the desired property is not too sparse without actually having to know anything about the distribution of the fractional parts of $hf'(K)$ as $h=1,2,\ldots,H$.
Then, in~\cref{eq:AveragedSum} one restricts to only those $n$ such that the fractional part of $f(n)+\xi_0$ is bounded away from~$1$ in such a way that no carry to a next integer occurs when adding the four terms
\[
	f(n) + \xi_0, \qquad
	hf'(K) - \xi_0,\qquad I_{n,h}
	\qquad 
	J_{n,h}.
\]
Clearly, these choices of $n$ and $h$ ensure that~\cref{eq:carry} holds with $\eta_{n,h,\xi_0} = 0$. 
Finally, one can see that the number of $n$ which had to be discarded from~\cref{eq:AveragedSum} is not too large. In~\cite{banks2009multiplicative} this is accomplished by bounding the discrepancy of the sequence of fractional parts of $f(n)$ as $n=1,2,\ldots,N$, using the Erd\H{o}s--Tur\'an inequality to translate this to a problem of estimating certain exponential sums, and estimating these sums using a standard application of the van~der Corput method.

The above argument, and in particular the additive split~\cref{eq:carry}, reduces the proof of \cref{lem:ShortSumBound}
to bounding double sums of the shape
\begin{equation}\label{eq:double sum}
	\sum_{u\in\mathscr{U}} \sum_{v\in\mathscr{V}} a_{u} b_{v} \psi_{x,y;\,p}(u+v), 
\end{equation}
where $\mathscr{U},\mathscr{V}$ are subsets of $\FF_{p}$ and $a_{u},b_{v}$  	($u\in\mathscr{U}, v\in\mathscr{V}$) are certain weights. We also recall~\cref{eq:AveragedSum} and the subsequent discussion about separating $n$ and $h$; the need for including the weights arises from the potential failure of, e.g., $\floor{f(n)+\xi_0}$ to produce only distinct values modulo~$p$ as $n$ varies.
 
\subsection{Concluding the proof} 
Given the above discussion, it is evident that, up to carrying out the technicalities which are, however, all readily found in~\cite{banks2009multiplicative}, \cref{thm:main:result} follows from \cref{lem:ShortSumBound}, and in turn \cref{lem:ShortSumBound} may be deduced from appropriate bounds for double sums~\cref{eq:double sum}.
Such a suitable bound is given in \cref{cor:DoubleSumBound-eps} below. 

\section{Technical Details}

\subsection{Bounds for certain double sums}
\label{sec:DoubleSumBound}

Here we prove a bound on the double sums~\cref{eq:double sum} which concludes the proof of \cref{thm:main:result}.
In fact, we first give a slightly more general bound:
\begin{lem}\label{lem:DoubleSumBound}
	Suppose that $p$ is prime and $\mathscr{U},\mathscr{V}$ are subsets of $\FF_{p}$ of cardinalities 
	$U= \#\mathscr{U}$ and $V = \#\mathscr{V}$.	Then, for an arbitrary fixed integer $k$,  for any complex 
	numbers $a_{u},b_{v}$  	($u\in\mathscr{U}, v\in\mathscr{V}$) and $x,y\in\FF_p$ with $y\neq 0$, 
	we have
	\begin{equation}\label{eq:DoubleSumBound}
		\sum_{u\in\mathscr{U}} \sum_{v\in\mathscr{V}} a_{u} b_{v} \psi_{x,y;\,p}(u+v) 
		\ll_k AB  U^{1-1/(2k)} (V^{1/2}p^{1/(2k)} + Vp^{1/(4k)})
	\end{equation}
	where 
	\[
		A = \max_{u\in\mathscr{U}} \abs{a_u}
		\mand
		B = \max_{v\in\mathscr{V}} \abs{b_v}.
	\]	
\end{lem}

\begin{proof}
	Denote the left hand side of~\cref{eq:DoubleSumBound} by $\fS$.
	Then we apply H{\"o}lder's inequality and subsequently extend the summation over $u\in\mathscr{U}$ to $u\in\FF_p$, getting
	\begin{align*}
		\fS^{2k} &
		\leq A^{2k} U^{2k-1} \sum_{u\in\mathscr{U}} \abs[\Big]{
			\sum_{v\in\mathscr{V}} b_v \psi_{x,y;\,p}(u+v)
		}^{2k} \\
		& \leq A^{2k} U^{2k-1} \sum_{u\in\FF_p} \abs[\Big]{
			\sum_{v\in\mathscr{V}} b_v \psi_{x,y;\,p}(u+v)
		}^{2k}.
	\end{align*}
	Therefore,
	\begin{align*}
		\fS^{2k} &
		\leq A^{2k} U^{2k-1} \sum_{u\in\FF_p} \mathop{\sum\cdots\sum}_{\substack{
			\boldsymbol{v} =(v_1,\ldots,v_k)\in\mathscr{V}^k \\
			\boldsymbol{w} =(w_1,\ldots,w_k)\in\mathscr{V}^k
		}} \prod_{r=1}^k \prod_{s=1}^k b_{v_r} \psi_{x,y;\,p}(u+v_r) \overline{b_{w_s} \psi_{x,y;\,p}(u+w_s)} \\ &
		= A^{2k} U^{2k-1} \mathop{\sum\cdots\sum}_{\substack{
			\boldsymbol{v} =(v_1,\ldots,v_k)\in\mathscr{V}^k \\
			\boldsymbol{w} =(w_1,\ldots,w_k)\in\mathscr{V}^k
		}}
		\prod_{r=1}^k b_{v_r}   \prod_{s=1}^k \overline{b_{w_s}}
		\sideset{}{^*}\sum_{u \in \FF_p} \e[p]{x \sum_{r=1}^k \sum_{s=1}^k v_r w_s + y R_{\boldsymbol{v},\boldsymbol{w}}(u)}[\bigg],
	\end{align*}
	where
	\[
		R_{\boldsymbol{v},\boldsymbol{w}}(X) = \sum_{r=1}^k \frac{1}{X+v_r} - \sum_{s=1}^k \frac{1}{X+w_s} \in \FF_p(X),
	\]
	and  $\smash{\sum^*}$ restricts the summation to those $u\in\FF_p$ such that \emph{all} of the numbers
	\[
		u+v_1,\ldots,u+v_k,u+w_1,\ldots,u+w_k \in \FF_p^*
	\]
	that is, are non-zero in  $\FF_p$.
	Thus,
	\begin{equation}\label{eq:UseBombieriWeilHere}
		\fS^{2k} 
		\leq A^{2k} B^{2k} U^{2k-1} \mathop{\sum\sum}_{\boldsymbol{v},\boldsymbol{w}\in\mathscr{V}^k} \abs[\Big]{
			\sideset{}{^*}\sum_{u} \e[p]{ R_{\boldsymbol{v},\boldsymbol{w}}(u) }
		}. 
	\end{equation}
	
	We can assume that $p> k$ as otherwise there is nothing to prove. 
	Examining the poles of the rational function $R_{\boldsymbol{v},\boldsymbol{w}}(X)$, we see 
	that it is constant, and in fact vanishes, only if the vectors $\boldsymbol{v}$ and $\boldsymbol{w}$
	differ only by a permutation of their components. This happens only for $O_k(V^k)$ choices of 
	$\boldsymbol{v},\boldsymbol{w}\in\mathscr{V}^k$. For such choices we estimate the inner most sum in~\cref{eq:UseBombieriWeilHere} trivially as $p$.
	Hence the total contribution to~\cref{eq:UseBombieriWeilHere} from such $\boldsymbol{v}, 
	\boldsymbol{w}\in\mathscr{V}^k$ is $O_k(V^kp)$. 
	
	For the remaining $O(V^{2k})$ choices  we use the Weil bound of exponential sums  with rational functions
	(see, for example,~\cite[Theorem~2]{moreno1991exponential-sums}, several more general bounds can also be found in~\cite{cochrane06})   
	and conclude that  the total contribution to~\cref{eq:UseBombieriWeilHere} from such $\boldsymbol{v}, 
	\boldsymbol{w}\in\mathscr{V}^k$ is $O_k(V^{2k}p^{1/2})$. 
	
	Hence, 
	\[
		\fS^{2k} 
		\ll_k  A^{2k} B^{2k} U^{2k-1} (V^kp + V^{2k}p^{1/2}),
	\]
	and the result follows. 
\end{proof}

\begin{cor}\label{cor:DoubleSumBound-eps}
	For any $\varepsilon>0$, in the setting of~\cref{lem:DoubleSumBound}, and assuming that
	\[
		   U \ge p^{1/2+\varepsilon}
				\mand
			V\ge p^{\varepsilon},
	\]
	we have
	\[
		\sum_{u\in\mathscr{U}}\sum_{v\in\mathscr{V}} a_{u} b_{v} \psi_{x,y;\,p}(u+v)
		\ll_{\varepsilon} ABUV  p^{-\varepsilon^2}.
	\]
\end{cor}
\begin{proof}
	Taking $k=\ceil{1/(2\varepsilon)}$ in~\cref{lem:DoubleSumBound}  we have 
	$V\ge p^{1/(2k)}$ and thus the right hand side of~\cref{eq:DoubleSumBound} can be replaced
	with 
	\[AB  U^{1-1/(2k)}  Vp^{1/(4k)} = ABUV \parentheses{p^{1/2} U^{-1}}^{1/(2k)}\le ABUV  p^{-\varepsilon/(2k)}.
	\] 
	Since $U \le p$, we have $\varepsilon\le1/2$ and thus one verifies that $2k \le  \varepsilon^{-1}$. 
	The result now follows.
\end{proof}

\subsection{Explicit version of \texorpdfstring{\cref{lem:ShortSumBound}}{Lemma\autoref{lem:ShortSumBound}}}
\label{sec:ExplicitLemma}

Before being able to address the explicit choice of $\delta$ given in~\cref{eq:delta}, we need an explicit version of \cref{lem:ShortSumBound}.
We remark that in the next result the saving $\smallDelta$ is \emph{independent} of $\kappa$.
\begin{lem}
	\label{lem:ShortSumBound:regularity}
	In the setting of~\cref{lem:ShortSumBound}, assuming $\varepsilon$ to be sufficiently small, one may take
	\(
		\smallDelta = \varepsilon^2 / 26
	\).
\end{lem}
\begin{proof}
	We extract only the relevant part of the proof of \cite[Theorem~5.1]{banks2009multiplicative} (adjusted to our setting): we have
	\[
		\sum_{K < n\leq K+L} \psi_{x,y;\,p}(\floor{f(n)})
		\ll_{\varepsilon,f} {\mathscr B} V^{-1}p^{o_f(1)} + Lp^{-\delta_1} + p^{\varepsilon/2},
	\]
	where 
	\begin{itemize}
		\item $\delta_1>0$ is from~\cite[Equation~(12)]{banks2009multiplicative} 
		\item ${\mathscr B}$ is the bound obtained from \cref{cor:DoubleSumBound-eps} applied with
		\[
			L \geq U = p^{1/2+\varepsilon+o_f(1)}
			\mand
			V= p^{\varepsilon/4+o_f(1)},
		\]
		and weights $a_u,b_u$ of size $p^{o_f(1)}$ as $p\to\infty$.
	\end{itemize}
	Therefore, by \cref{cor:DoubleSumBound-eps} and~\cref{eq:ShortSumBound:Ranges},
	\[
		\sum_{K < n\leq K+L} \psi_{x,y;\,p}(\floor{f(n)})
		\ll_{\varepsilon,f} L p^{-(\varepsilon/5)^2 + o_f(1)} + Lp^{-\delta_1} + Lp^{-1/2-\varepsilon/2}.
	\]
	Now only the value of $\delta_1$ needs closer inspection. Looking at the lines that precede~\cite[Equation~(12)]{banks2009multiplicative}, and recalling~\cref{eq:ShortSumBound:Ranges}, one may check that, for $\varepsilon$ sufficiently small, $\delta_1 = \varepsilon^2$ is admissible. Upon plugging this in the above bound, the result follows.
\end{proof}

\subsection{Proof of \texorpdfstring{\cref{thm:main:result}}{{Theorem\autoref{thm:main:result}}}}
\label{sec:main proof}
Assume that $\varepsilon>0$ is sufficiently small. 
Then the inequalities~\cref{eq:N:and:p:relations:i_to_iii} are all implied by~\cref{eq:N:and:p:relations:iv}, and the latter is clearly satisfied when choosing
\[
	c = \frac{\varepsilon \kappa^2}{1+6\kappa+2\varepsilon\kappa} < 1.
\]
A close inspection of the proof of \cite[Theorem~6.1]{banks2009multiplicative} shows that
\[
	\sum_{n=1}^N \psi_{x,y;\,p}(\floor{f(n)}) \ll_{\varepsilon,f} N p^{-\delta_1} + Np^{-c},
\]
where $\delta_1$ is now any admissible exponent $\delta$ of $p$ from the bound obtained from \cref{lem:ShortSumBound} with $\varepsilon$ replaced by $c$. In particular, by \cref{lem:ShortSumBound:regularity}, we may choose $\delta_1 = c^2 / 26$.
This shows that in~\cref{eq:main:result},  again assuming $\varepsilon$ to be sufficiently small, 
we may take $\delta$ as in~\cref{eq:delta}, which concludes the proof of \cref{thm:main:result}.

\subsection{Proofs of \texorpdfstring{\cref{cor:DistrInt} and \cref{cor:Exist}}{Corollary\autoref{cor:DistrInt} and Corollary\autoref{cor:Exist}}}

\cref{cor:DistrInt} follows at once if one combines \cref{thm:main:result} with the Erd\H{o}s--Tur\'{a}n inequality (see, for instance,~\cite[Theorem~1.21]{Drm-Tichy-UD}).
 
To prove \cref{cor:Exist}, we note that $HN > p^{2-\xi}$ implies that $H, N >p^{1-\xi}$.
In particular, assuming $\xi < 1/4$ we see that $N$ satisfies the necessary lower bound in~\cref{eq:ineq N}.
We now  define 
\[
	N_0 = \min\braces{ N, \ceil{p^{1/(2-\kappa)}}-1 }
\]
thus $T_f(N,\scrI) \ge T_f(N_0,\scrI)$ and we also see that \cref{cor:DistrInt}
applies to  $T_f(N_0,\scrI)$. Taking 
\[
	\xi< 1-  1/(2-\kappa)
\]
we see that we can assume that $N_0 =  \ceil{p^{1/(2-\kappa)}} -1$.
Since $H > p^{1-\xi}$ we have the desired result.

\section{Comments}

\subsection{Some predecessors of our approach}
Results of the shape of \cref{cor:DoubleSumBound-eps} in the case Dirichlet characters appear to have been developed by Karatsuba~\cite{karatsuba91} building on earlier work of Davenport and Erd\H{o}s~\cite[Lemma~3]{davenport1952the-distribution} and Burgess~\cite[Lemma~2]{burgess1962on-character-sums}. Indeed, the proof of our \cref{cor:DoubleSumBound-eps} proceeds in a similar vein and the averaging procedure underlying~\cref{eq:AveragedSum} has been used extensively (again, see~\cite{karatsuba91} and the references therein).

The method has been subsequently adapted in a series of works~\cite{banks2006nonresidues,banks2006short-character-sums,banks2007prime-divisors} on properties of Beatty sequences, which often amounts to studying sums of the shape~\cref{eq:CharacterSum} with $f(t) = \alpha t + \beta$ (with an  irrational  $\alpha>1$ and real $\beta$) and $\chi$ potentially replaced with some other function of arithmetic interest. Here, in contrast to the situation in~\cref{eq:2ndDerivativeGrowth} which opens up the possibility of using the van~der Corput method, the quality of error terms generally also depends on the Diophantine properties of $\alpha$.

\subsection{Further problems}

We have not made any attempt at optimising the value of~$\delta$ for which one can prove \cref{thm:main:result}. It would be interesting to see how large a value of~$\delta$ the method from~\cite{banks2009multiplicative} can produce if all parameters are chosen optimally. Likewise, ascertaining the sharpest form of \cref{cor:Exist} would also be interesting.

Moreover, any improvement on \cref{lem:DoubleSumBound}, even for narrower ranges of $U$ and $V$, would be of independent interest.
As a first step in this direction we record the following result which is non-trivial whenever $UV\gg p$ and, in this range and up to the implied constants, at least as good as \cref{lem:DoubleSumBound} and
stronger when $U$ and $V$ are both large:
\begin{prop}
	In the setting of \cref{lem:DoubleSumBound}, we have
	\[
		\abs[\Big]{
			\sum_{u\in\mathscr{U}} \sum_{v\in\mathscr{V}} a_{u}b_{v} \psi_{x,y;\,p}(u+v)
		}
		\ll ABU^{1/2}V^{1/2}  p^{1/2},
	\]
	where the implied constant is absolute.
\end{prop}
\begin{proof}
	We keep the notation $\fS$ from the proof of~\cref{lem:DoubleSumBound}. Then
	\begin{align*}
		\fS &
		= \sum_{w \in \FF_p} \psi_{x,y;\,p}(w) \sum_{u\in\mathscr{U}} \sum_{v\in\mathscr{V}} a_{u}b_{v} \frac{1}{p} \sum_{\lambda \in \FF_p} \e[p]{ \lambda(u+v-w) } \\ &
		= \frac{1}{p} \sum_{\lambda \in \FF_p} \sum_{w \in \FF_p} \psi_{x-\lambda,y;\,p}(w) \sum_{u\in\mathscr{U}} a_{u} \e[p]{ \lambda u }	 \sum_{v\in\mathscr{V}} b_{v} \e[p]{\lambda v}.
	\end{align*}
	From the Weil bound (for the sum over $w$) and using Cauchy's inequality, we infer
		\[
		\fS^2 \ll \frac{1}{\sqrt{p}} \sum_{\lambda \in \FF_p}\abs[\Big]{ \sum_{u\in\mathscr{U}} a_{u} \e[p]{ \lambda u }}^2 \cdot \frac{1}{\sqrt{p}} \sum_{\lambda \in \FF_p}\abs[\Big]{ \sum_{v\in\mathscr{V}} b_{v} \e[p]{ \lambda v }}^2.
	\]
	Upon expanding the square in both sums and using orthogonality of characters, this yields
	\[
		\fS \ll (\sqrt{p} U A^2)^{1/2} \, (\sqrt{p} V B^2)^{1/2} = ABU^{1/2}V^{1/2}  p^{1/2},
	\]
	which gives the result.
\end{proof}

\bibliographystyle{abbrv}
\bibliography{KloostTwiceDiff}

\begin{thebibliography}{10}

\bibitem{akbal2017}
Y.~{Akbal}.
\newblock Friable values of {P}iatetski-{S}hapiro sequences.
\newblock {\em Proc. Amer. Math. Soc.}, 145(10):4255--4268, 2017.

\bibitem{akbal2018}
Y.~{Akbal}.
\newblock {Rough values of Piatetski-Shapiro sequences}.
\newblock {\em {Monatsh. Math.}}, 185:1--15, 2018.

\bibitem{akbal2015}
Y.~{Akbal} and A.~M. {G\"ulo\u{g}lu}.
\newblock {Piatetski-Shapiro meets Chebotarev}.
\newblock {\em {Acta Arith.}}, 167:301--325, 2015.

\bibitem{akbal2016}
Y.~{Akbal} and A.~M. {G\"ulo\u{g}lu}.
\newblock {Waring's problem with Piatetski-Shapiro numbers}.
\newblock {\em {Mathematika}}, 62:524--550, 2016.

\bibitem{baker2015}
R.~C. {Baker} and W.~D. {Banks}.
\newblock {Character sums with Piatetski-Shapiro sequences}.
\newblock {\em {Q. J. Math.}}, 66:393--416, 2015.

\bibitem{baker2013}
R.~C. {Baker}, W.~D. {Banks}, J.~{Br\"udern}, I.~E. {Shparlinski}, and A.~J.
  {Weingartner}.
\newblock {Piatetski-Shapiro sequences}.
\newblock {\em {Acta Arith.}}, 157:37--68, 2013.

\bibitem{baker2014}
R.~C. {Baker}, W.~D. {Banks}, V.~Z. {Guo}, and A.~M. {Yeager}.
\newblock {Piatetski-Shapiro primes from almost primes}.
\newblock {\em {Monatsh. Math.}}, 174:357--370, 2014.

\bibitem{banks2008}
W.~D. {Banks}, M.~Z. {Garaev}, D.~R. {Heath-Brown}, and I.~E. {Shparlinski}.
\newblock {Density of non-residues in Burgess-type intervals and applications}.
\newblock {\em {Bull. Lond. Math. Soc.}}, 40:88--96, 2008.

\bibitem{banks2016}
W.~D. {Banks}, V.~Z. {Guo}, and I.~E. {Shparlinski}.
\newblock {Almost primes of the form $\lfloor p^c \rfloor$}.
\newblock {\em {Indag. Math., New Ser.}}, 27:423--436, 2016.

\bibitem{banks2006nonresidues}
W.~D. {Banks} and I.~E. {Shparlinski}.
\newblock {Non-residues and primitive roots in Beatty sequences}.
\newblock {\em Bull. Aust. Math. Soc.}, 73:433--443, 2006.

\bibitem{banks2006short-character-sums}
W.~D. {Banks} and I.~E. {Shparlinski}.
\newblock {Short character sums with Beatty sequences}.
\newblock {\em Math. Res. Lett.}, 13:539--547, 2006.

\bibitem{banks2007prime-divisors}
W.~D. {Banks} and I.~E. {Shparlinski}.
\newblock {Prime divisors in Beatty sequences}.
\newblock {\em J. Number Theory}, 123:413--425, 2007.

\bibitem{banks2009multiplicative}
W.~D. {Banks} and I.~E. {Shparlinski}.
\newblock {Multiplicative character sums with twice-differentiable functions}.
\newblock {\em {Q. J. Math.}}, 60:401--411, 2009.

\bibitem{burgess1962on-character-sums}
D.~A. {Burgess}.
\newblock {On character sums and primitive roots}.
\newblock {\em Proc. London Math. Soc. (3)}, 12:179--192, 1962.

\bibitem{cochrane06}
T.~Cochrane and C.~Pinner.
\newblock Using {S}tepanov's method for exponential sums involving rational
  functions.
\newblock {\em J. Number Theory}, 116:270--292, 2006.

\bibitem{davenport1952the-distribution}
H.~{Davenport} and P.~{Erd\H{o}s}.
\newblock {The distribution of quadratic and higher residues}.
\newblock {\em Publ. Math. Debrecen}, 2:252--265, 1952.

\bibitem{Drm-Tichy-UD}
M.~{Drmota} and R.~F. {Tichy}.
\newblock {\em Sequences, discrepancies and applications}, volume 1651 of {\em
  Lecture Notes in Mathematics}.
\newblock Springer-Verlag, Berlin, 1997.

\bibitem{karatsuba91}
A.~A. {Karatsuba}.
\newblock Distribution of values of {D}irichlet characters on additive
  sequences.
\newblock {\em Dokl. Akad. Nauk SSSR}, 319:543--545, 1991.

\bibitem{liu2017}
K.~{Liu}, I.~E. {Shparlinski}, and T.~{Zhang}.
\newblock {Squares in Piatetski-Shapiro sequences}.
\newblock {\em {Acta Arith.}}, 181:239--252, 2017.

\bibitem{lu2018}
{\relax Ya}.~M. {Lu}.
\newblock {An additive problem on Piatetski-Shapiro primes}.
\newblock {\em {Acta Math. Sin., Engl. Ser.}}, 34:255--264, 2018.

\bibitem{mauduit1995repartition}
{\relax Ch}.~{Mauduit} and J.~{Rivat}.
\newblock R\'{e}partition des fonctions {$q$}-multiplicatives dans la suite
  {$([n^c])_{n\in\mathbb{N}},\ c>1$}.
\newblock {\em Acta Arith.}, 71:171--179, 1995.

\bibitem{mirek2015}
M.~{Mirek}.
\newblock {Roth's theorem in the Piatetski-Shapiro primes}.
\newblock {\em {Rev. Mat. Iberoam.}}, 31:617--656, 2015.

\bibitem{moreno1991exponential-sums}
C.~J. {Moreno} and O.~{Moreno}.
\newblock {Exponential sums and Goppa codes. {I}}.
\newblock {\em Proc.\ Amer.\ Math.\ Soc.}, 111:523--531, 1991.

\bibitem{muellner2017normality}
C.~{M\"ullner} and L.~{Spiegelhofer}.
\newblock {{N}ormality of the {T}hue--{M}orse sequence along
  {P}iatetski-{S}hapiro sequences, {II}}.
\newblock {\em Israel J.\ Math.}, 220:691--738, 2017.

\bibitem{pyateckii-sapiro1953on-the-distribution}
I.~I. {Pyatecki\u{\i}-\v{S}apiro}.
\newblock {On the distribution of prime numbers in sequences of the form
  $[f(n)]$}.
\newblock {\em Mat.\ Sbornik N.S.}, 33(75):559--566, 1953.

\bibitem{rivat2001prime-numbers}
J.~{Rivat} and W.~{Jie}.
\newblock {Prime numbers of the form $[n^c]$}.
\newblock {\em {Glasg. Math. J.}}, 43:237--254, 2001.

\bibitem{rivat2001nombres-premiers}
J.~{Rivat} and P.~{Sargos}.
\newblock {Nombres premiers de la forme $\lfloor n^c\rfloor$}.
\newblock {\em Canad.\ J.\ Math.}, 53:414--433, 2001.

\bibitem{spiegelhofer2014piatetski-shapiro}
L.~{Spiegelhofer}.
\newblock {Piatetski-Shapiro sequences via Beatty sequences}.
\newblock {\em Acta Arith.}, 166:201--229, 2014.

\bibitem{technau2018modular-hyperbolas}
M.~{Technau}.
\newblock {Modular hyperbolas and Beatty sequences}, 2018.
\newblock Preprint:
  \href{https://arxiv.org/abs/1808.00413}{arXiv:1808.00413~[math.NT]}.

\end{thebibliography}
\end{document}